\documentclass[reqno,11pt]{amsart}
\usepackage{amsfonts}
\usepackage{a4wide}
\usepackage{color}
\usepackage{mathrsfs}
\usepackage{mathtools}
\usepackage{amsmath}
\usepackage{amssymb}
\usepackage{bbm}
\usepackage{bm}
\usepackage{esint}
\usepackage{nicefrac}
\usepackage{comment}
\numberwithin{equation}{section}
\usepackage[colorlinks,citecolor=green,linkcolor=red]{hyperref}

\usepackage[latin1]{inputenc}
\input xy
\xyoption{all}

\newtheorem{theorem}{Theorem}[section]
\newtheorem{Ca}[theorem]{Corollary}
\newtheorem{Th}[theorem]{Theorem}
\newtheorem{Lm}[theorem]{Lemma}

\newtheorem{Def}[theorem]{Definition}

\newtheorem{Remark}[theorem]{Remark}
\newtheorem{Problem}[theorem]{Problem}

\title[Nonexistence of Linear Extensions for Besov Traces]
{On the Linearity of Extension Operators for Traces of Besov Spaces on Metric Measure Spaces: The Limiting Case}
\author{Aleksei Y. Chikalov}
\address{Steklov Mathematical Institute of Russian Academy of Sciences}
\email{chikalov.a@phystech.edu}
\begin{document}
\allowdisplaybreaks
\subjclass[2010]{46E35, 46B25, 46B45, 42B35}
\keywords{Besov spaces, traces, metric measure spaces, Ahlfors--David regular sets, hyperbolic fillings, extension operators}
\begin{abstract}
Let $p\in[1,\infty)$, and let $X=(X,d,\mu)$ be a metric measure
space such that $\mu$ is uniformly locally doubling and $X$ supports
a local weak $(1,p)$-Poincar\'e inequality. Given $\theta\in (0, p)$ and an Ahlfors--David codimension-$\theta$ regular set $E\subset X$, the trace-space of the Besov space $B^{\theta/p}_{p, 1}(X)$ to $E$ can be identified with $L_p(E, \mathcal{H}_{\theta}\lfloor_E)$. If there exists a measurable set $A\subset E$ with
$0<\mathcal H_\theta(A)<\infty$ such that $\mathcal H_\theta\lfloor_A$ is nonatomic, we prove that there is no bounded linear extension operator $\operatorname{Ext}:L_p(E, \mathcal{H}_{\theta}\lfloor_E)\to B^{\theta/p}_{p, 1}(X)$.

\end{abstract}
\maketitle
\tableofcontents
\section{Introduction}
\noindent
The theory of function spaces plays a central role in modern analysis. One of its fundamental problems is the \emph{trace problem}: given a function space on an ambient space, one seeks a \emph{sharp intrinsic description} of the restrictions of its elements to various closed subsets. In recent years, function spaces on \emph{metric measure spaces} have attracted considerable attention. In the present paper, we focus on Besov spaces in this setting. For background on function spaces on metric measure spaces, and in particular on Besov spaces, we refer the reader to \cite{Shanm, AKZ, Han, sob_mms, Koskela, soto_bes} and the references therein.
\par
The trace problem goes back to the seminal work of Gagliardo \cite{gagl}. He showed that, given $p\in (1, \infty)$, the trace-space of the Sobolev space $W^1_p(\mathbb{R}^n)$ to the hyperplane $\mathbb{R}^{n-1}\subset\mathbb{R}^n$ is linearly and continuously isomorphic to the Besov space $B^{1-1/p}_{p, p}(\mathbb{R}^{n-1})$. Moreover, there exists a bounded linear extension operator $\operatorname{Ext}:B^{1-1/p}_{p, p}(\mathbb{R}^{n-1})\to W^1_p(\mathbb{R}^n)$.
The case $p=1$ is exceptional. In this limiting case, the trace-space of $W_1^1(\mathbb{R}^n)$ to $\mathbb{R}^{n-1}$ loses its smoothness and becomes $L_1(\mathbb{R}^{n-1})$. A further remarkable feature of this endpoint phenomenon is that a bounded extension operator has to be nonlinear (see \cite{kaz_woj, Peetre}).
\par
An analogous limiting phenomenon occurs for Besov spaces. In his pioneering work \cite{Bes_or}, Besov proved that the trace-space of $B^s_{p, q}(\mathbb{R}^n)$ to $\mathbb{R}^{d}$, where $p, q \in [1, \infty]$, $d\in [1, n-1]$, and $s>\frac{n-d}{p}$, is linearly and continuously isomorphic to
$B^{s-(n-d)/p}_{p, q}(\mathbb{R}^{d})$. In this nonlimiting range, there exists a bounded linear extension operator $\operatorname{Ext}:B^{s-(n-d)/p}_{p, q}(\mathbb{R}^{d})\to B^s_{p, q}(\mathbb{R}^n)$.
The limiting case $p\in [1, \infty)$, $q=1$, and $s=\frac{n-d}{p}$, is substantially different. In the codimension-one case \(d=n-1\), Burenkov and Gol'dman proved in \cite{bur_gold} that the trace-space of \(B^{1/p}_{p,1}(\mathbb R^n)\) to \(\mathbb R^{n-1}\) can be identified with \(L_p(\mathbb R^{n-1})\). They also showed that no bounded linear extension operator exists in this case. In a separate paper \cite{Gol}, Gol'dman extended the trace-space identification to arbitrary linear subspaces \(\mathbb R^d\subset\mathbb R^n\), where \(1\le d\le n-1\), but did not address the linearity of the corresponding extension operator.
\par
In a recent work \cite{chik}, the author extended this classical limiting trace theorem for Besov spaces to the setting of metric measure spaces. More precisely, under standard regularity assumptions on the metric measure space \((X,d,\mu)\) (see Section~2.1 for details), the trace-space of \(B^{\theta/p}_{p,1}(X)\), where \(p\in[1,\infty)\) and \(\theta\in(0,p)\), to an Ahlfors--David codimension-\(\theta\) regular set \(E\) can be identified with $L_p(E, \mathcal{H}_{\theta}\lfloor_E)$.
The extension operator constructed in \cite{chik} is nonlinear. \emph{The aim of the present paper} is to show that this is not merely an artifact of the construction. Under mild additional assumptions on \(E\), we prove that no bounded linear extension operator $\operatorname{Ext}:L_p(E, \mathcal{H}_{\theta}\lfloor_E) \to B^{\theta/p}_{p, 1}(X)$
exists.
\par
The classical proofs of the corresponding nonlinearity result by Burenkov and Gol'dman rely heavily on the linear and Fourier-analytic structure of the Euclidean subspace \(\mathbb R^{n-1}\). It is not clear how to adapt their approximation arguments to general Ahlfors--David regular sets, even in \(\mathbb R^n\), and still less to subsets of metric measure spaces. We therefore use a different approach, combining hyperbolic fillings with Banach-space arguments.

\subsection{Main result}
In order to formulate the problem and state the main result precisely, we first fix some notation and recall several basic definitions.
\par
By a \emph{metric measure space} we mean a triple
\(X=(X,d,\mu)\), where \((X,d)\) is a \emph{complete separable metric space} and
\(\mu\) is a \emph{Borel regular} measure giving \emph{positive finite mass} to every ball.
We impose the following additional regularity assumptions.
Given \(\sigma\in[1,\infty)\), we say that $X=(X, d, \mu)$ is \emph{$\sigma$-admissible} if:
\begin{enumerate}
    \item \(\mu\) is \emph{uniformly locally doubling};
    \item \((X,d,\mu)\) supports a \emph{local weak \((1,\sigma)\)-Poincar\'e inequality}.
\end{enumerate}
The precise forms of these assumptions are given in Section~2.1.
\par
We next recall the mean oscillation definition of Besov spaces used in \cite{chik}. There are several equivalent approaches to Besov spaces on metric measure spaces. Although the main argument of the present paper will use a discrete description via hyperbolic fillings, recalled in Section~2.2, it is useful to start with the oscillation formulation in order to connect the notation with the trace theorem from \cite{chik}.
Throughout the paper, $B_r(x)$ denotes the open ball centered at \(x\in X\) with radius \(r>0\). For each $f\in L_1^{\operatorname{loc}}(X)$ and $t>0$, we set
\begin{equation}
    \overline{\Delta}_tf(x) := \frac{1}{\mu(B_t(x))}\int\limits_{B_t(x)}\left|f(y) - \frac{1}{\mu(B_t(x))}\int\limits_{B_t(x)}f(z)d\mu(z)\right|d\mu(y).
\end{equation}
Given $p, q\in [1, \infty]$ and $s\in (0, 1)$, the Besov space $B^s_{p, q}(X)$ consists of all $f\in L_1^{\operatorname{loc}}(X)$ for which the following norm is finite:
\begin{equation}
    \|f\|_{B^s_{p, q}(X)} := \|f\|_{L_p(X)} + \left(\int\limits_0^1 \left(t^{-s}\|\overline{\Delta}_tf\|_{L_p(X)} \right)^{q}\frac{dt}{t} \right)^{1/q}<\infty
\end{equation}
with the standard modification when $q=\infty$.
\par
We use the codimensional Hausdorff measure $\mathcal{H}_{\theta}$ on $X$ (its construction is recalled in Section~2.1). A closed set $E\subset X$ is said to be \emph{Ahlfors--David} codimension-$\theta$ regular, $\theta\in [0, \infty)$, if there exist constants $C_1, C_2>0$ such that 
\begin{equation}
    C_1\frac{\mu(B_r(x))}{r^{\theta}} \le \mathcal{H}_{\theta}(B_r(x)\cap E)\le C_2\frac{\mu(B_r(x))}{r^{\theta}}, \qquad \text{for all } (x, r) \in E\times (0, 1].
\end{equation}
Let $\theta \in (0, \infty)$, let $E\subset X$ be an Ahlfors--David codimension-$\theta$ regular set, and let $f\in L_1^{\operatorname{loc}}(X)$. A measurable function $\phi:E\to\mathbb{R}$ is called a \emph{trace} of $f$ if
\begin{equation}
\label{eq.tr_def}
    \lim_{r\to0}\frac{1}{\mu(B_r(x))}\int\limits_{B_r(x)}|f(y)-\phi(x)|d\mu(y) = 0, \qquad \text{for $\mathcal{H}_{\theta}$-a.e. } x\in E.
\end{equation}
In this case, the equivalence class of $\phi$ modulo $\mathcal{H}_{\theta}\lfloor_E$-negligible sets is denoted by $\operatorname{Tr}f$.
\par
By \cite[Theorem~1.2]{chik}, if $p\in [1, \infty)$, $\theta \in (0, p)$, and $E\subset X$ is Ahlfors--David codimension-$\theta$ regular, then every $f\in B^{\theta/p}_{p, 1}(X)$ has a trace to $E$ in the sense of \eqref{eq.tr_def}. Hence we can define the \emph{trace-space}
\begin{equation}
    B^{\theta/p}_{p, 1}(X)\big|_E := \left\{\operatorname{Tr}f: f\in B^{\theta/p}_{p, 1}(X)\right\}.
\end{equation}
We equip this space with the \emph{quotient-space norm}, i.e.,
\begin{equation}
    \|\phi\|_{B^{\theta/p}_{p, 1}(X)\big|_E}:=\inf\left\{\|f\|_{B^{\theta/p}_{p, 1}(X)}: \operatorname{Tr}f = \phi\right\}, \qquad \phi \in B^{\theta/p}_{p, 1}(X)\big|_E.
\end{equation}
The corresponding \emph{trace operator} is
\begin{equation}
    \operatorname{Tr}:B^{\theta/p}_{p, 1}(X) \to B^{\theta/p}_{p, 1}(X)\big|_E, \qquad f\mapsto \operatorname{Tr}f.
\end{equation}
The trace problem can be stated as follows.
\begin{Problem}
    Let $(X, d, \mu)$ be $p$-admissible, $p\in [1, \infty)$, and $E\subset X$ be Ahlfors--David codimension-$\theta$ regular, $\theta\in (0, p)$.
    \begin{enumerate}
        \item Given $\phi:E\to\mathbb{R}$, find necessary and sufficient conditions for the existence of an extension of $\phi$, that is, a function $f\in B^{\theta/p}_{p, 1}(X)$ such that $\operatorname{Tr}f=\phi$;
        \item Find an intrinsic norm on the trace-space $B^{\theta/p}_{p, 1}(X)\big|_E$ which is equivalent to the quotient-space norm;
        \item Decide whether there exists a bounded operator, called an extension operator,
        \begin{equation}
            \operatorname{Ext}:B^{\theta/p}_{p, 1}(X)\big|_E\to B^{\theta/p}_{p, 1}(X),
        \end{equation}
        which is a right inverse to the trace operator
        \begin{equation}
            \operatorname{Tr} \circ \operatorname{Ext} = \operatorname{Id} \qquad \text{on } B^{\theta/p}_{p, 1}(X)\big|_E.
        \end{equation}
        \item Determine whether such an extension operator can be chosen linear.
    \end{enumerate}
\end{Problem}
In \cite{chik}, the first three questions were answered. More precisely, it was proved that $B^{\theta/p}_{p, 1}(X)\big|_E = L_p(E, \mathcal{H}_{\theta} \lfloor_E)$ with equivalent norms. Moreover, a bounded nonlinear extension operator
\begin{equation}
    \operatorname{Ext}:L_p(E, \mathcal{H}_{\theta}\lfloor_E) \to B^{\theta/p}_{p, 1}(X)
\end{equation}
was constructed. In the present paper, we address the last question.
\par
\textit{The main result} is the following.
\begin{Th}
    \label{Th.main_stat}
    Let $p\in [1, \infty)$, let $\theta \in (0, p)$, and let $(X, d, \mu)$ be $p$-admissible. Assume that $E\subset X$ is Ahlfors--David codimension-$\theta$ regular. Suppose, in addition, that there exists a measurable set $A\subset E$ such that $0<\mathcal{H}_{\theta}(A)<\infty$ and $\mathcal{H}_{\theta}\lfloor_A$ is nonatomic. Then there is no bounded linear extension operator
    \begin{equation}
        \operatorname{Ext}: L_p(E, \mathcal{H}_{\theta}\lfloor_E) \to B^{\theta/p}_{p, 1}(X).
    \end{equation}
\end{Th}
\par
Let us briefly explain the idea of the proof and comment on the additional assumption on $E$ in Theorem~\ref{Th.main_stat}. When $p\neq 2$, the argument follows the strategy of \cite{kaz_woj}. More precisely, using the discrete description of Besov spaces via hyperbolic fillings, we embed $B^{\theta/p}_{p, 1}(X)$ isomorphically into $\ell_1(\ell_p)$ (see Section~2.2). On the other hand, if $\mathcal{H}_{\theta}\lfloor_A$ is nonatomic and $0<\mathcal{H}_{\theta}(A)<\infty$, then the Rademacher functions generate a closed subspace of $L_p(A, \mathcal{H}_{\theta}\lfloor_A)$ isomorphic to $\ell_2$. Hence a bounded linear extension operator would yield an isomorphic embedding $\ell_2 \hookrightarrow\ell_1(\ell_p)$, which is impossible for $p\neq 2$. The case $p=2$ is more delicate, since $\ell_2$ does embed into $\ell_1(\ell_2)$, for instance as the first row. In this case, the contradiction cannot come from an abstract Banach-space obstruction alone. Instead, we use an additional finite-scale argument which takes into account the interaction between the hyperbolic filling and the trace operator.
\par
The assumption that \(\mathcal H_\theta\lfloor_A\) has a nonatomic part of positive finite measure is essential for our method, since it guarantees the presence of a closed copy of $\ell_2$ inside the trace-space. At the same time, this assumption is rather mild. For instance, in the Euclidean case $X = \mathbb{R}^n$, if $E \subset \mathbb{R}^n$ is Ahlfors--David codimension-$\theta$ regular with $\theta<n$, then $\mathcal{H}_{\theta}\lfloor_E$ is nonatomic (see the discussion in Section~3). Thus, in the Euclidean setting, our result recovers and extends the nonlinearity phenomenon proved in \cite{bur_gold}. 
\subsection{Plan of the paper}
The paper is organized as follows.
\begin{itemize}
\item In Section~2, we fix notation, recall the required definitions, and prove several auxiliary results. In particular, we discuss the relation between the hyperbolic filling approach and the mean oscillation approach to Besov spaces.
\item In Section~3, we prove the main result for \(p\neq2\).
\item In Section~4, we prove the main result for \(p=2\).
\end{itemize}
\par
\textbf{Acknowledgments}. The author would like to express his sincere gratitude to his advisor, Alexander I. Tyulenev, for suggesting the problem and for his guidance and support. The author is also grateful to Sergey V. Kislyakov for pointing out the Banach-space obstruction concerning embeddings into \(\ell_1(\ell_p)\) (Theorem~\ref{Th.keystone_prop}) and for explaining the idea behind the corresponding argument.

\section{Preliminaries}
Throughout the paper, the symbol $C$ denotes a positive constant whose value may change from line to line. If the dependence on parameters is important, we indicate it by writing $C(a, b, c, \ldots)$. We write $A\lesssim B$ if $A\le CB$. If both $A\lesssim B$ and $B \lesssim A$, we write $A  \approx B$.

\subsection{Geometric analysis background}
Let $(X, d)$ be a \emph{complete separable metric space}. For each $x\in X$ and $r>0$, we denote by $B_r(x)$ the open ball centered at $x$ with radius $r$. Given a ball $B = B_r(x)$ and a number $c>0$, we write $cB:=B_{cr}(x)$.
\par
The collection of all Lipschitz continuous functions $f:X\to\mathbb{R}$ is denoted by $\operatorname{LIP}(X)$. For each $f\in \operatorname{LIP}(X)$, we write $\operatorname{Lip}f$ for its global Lipschitz constant and $\operatorname{lip}f$ for its local Lipschitz constant, defined by
\begin{equation}
    \operatorname{lip}f(x) = \begin{cases}
        \limsup\limits_{y\to x} \frac{|f(y)-f(x)|}{d(y, x)}, \quad \text{if $x$ is an accumulation point},\\
        0, \quad \text{if $x$ is isolated}.
    \end{cases}
\end{equation}
\par
By a \emph{measure} $\mu$ on $X$, we mean a Borel regular nonzero locally finite outer measure on $X$. Given $p \in [1, \infty]$ and a $\mu$-measurable set $G\subset X$, we denote by $L_p(G, \mu)$ the space of all (equivalence classes of) $\mu$-measurable functions $f:G\to\mathbb{R}$ with finite norm
\begin{equation}
    \|f\|_{L_p(G, \mu)}:=\left(\int\limits_{G}|f(y)|^pd\mu(y)\right)^{1/p}
\end{equation}
with the usual modification when $p=\infty$. As usual, $L_p^{\operatorname{loc}}(X, \mu)$ is the collection of all (equivalence classes of) $\mu$-measurable functions $f:X\to\mathbb{R}$ such that $f\in L_p(B_r(x), \mu)$ for all $(x, r) \in X\times(0, \infty)$. Whenever \(0<\mu(B)<\infty\), we set, for each $f \in L_1^{\operatorname{loc}}(X, \mu)$,
\begin{equation}
    f_{B} :=\fint\limits_{B}f(y)d\mu(y) := \frac{1}{\mu(B)}\int\limits_{B}f(y)d\mu(y).
\end{equation}
\par
A \emph{metric measure space} is a triple $(X, d, \mu)$ , where \((X,d)\) is a complete separable metric space and
\(\mu\) is a measure on $X$ with $\operatorname{supp}\mu = X$. In particular, every ball has positive finite \(\mu\)-measure. In analysis on metric measure spaces, the following additional assumptions are standard.
\begin{Def}
    Given \(\sigma\in[1,\infty)\), we say that a metric measure space $X=(X, d, \mu)$ is $\sigma$-admissible and write \(X\in\mathfrak U_\sigma\) if:
    \begin{enumerate}
    \item \(\mu\) is uniformly locally doubling, i.e., for every \(R>0\)
    there exists \(C_\mu(R)>0\) such that
    \begin{equation}
        \mu(B_{2r}(x))\le C_\mu(R)\mu(B_r(x)),
        \qquad \text{for all } (x, r) \in X\times (0, R];
    \end{equation}
    \item \((X,d,\mu)\) supports a local weak \((1,\sigma)\)-Poincar\'e inequality, i.e., for every \(R>0\) there exist constants \(C_P(R)>0\) and
    \(\lambda(R)\ge1\) such that, for every \(f \in \operatorname{LIP}(X)\),
    \begin{equation}
        \fint\limits_{B_r(x)} |f(y)-f_{B_r(x)}|\,d\mu(y)
        \le
        C_P(R)r
        \left(
        \fint\limits_{B_{\lambda(R)r}(x)}
        (\operatorname{lip} f(y))^\sigma\,d\mu(y)
        \right)^{1/\sigma}
    \end{equation}
    for all \(x\in X\) and \(0<r\le R\).
\end{enumerate}
\end{Def}
\par
In general, the measure of balls in a metric measure space need not have a uniform power-law behavior. For this reason, it is often natural to use \emph{codimensional analogues} for the usual  Hausdorff measures. More precisely, given \(\theta\in[0,\infty)\), \(E\subset X\), and \(\delta\in(0,\infty]\), we put
\begin{equation}
\mathcal H_{\theta,\delta}(E)
:=
\inf
\left\{
\sum_i
\frac{\mu(B_{r_i}(x_i))}{r_i^\theta}
:
E\subset\bigcup_i B_{r_i}(x_i),
\quad
0<r_i<\delta
\right\},
\end{equation}
where the infimum is taken over all finite or countable coverings of \(E\) by balls $\{B_{r_i}(x_i)\}$. The \emph{codimension-\(\theta\) Hausdorff} measure is defined by
\begin{equation}
\mathcal H_\theta(E)
:=
\lim_{\delta\to0}\mathcal H_{\theta,\delta}(E).
\end{equation}
\par
Finally, we record the regularity condition on the trace set used throughout the paper.
\begin{Def}
    Given $\theta \in [0, \infty)$, we say that a closed set $E\subset X$ is Ahlfors--David codimension-$\theta$ regular if there exist constants $C_1, C_2>0$ such that
    \begin{equation}
        C_1\frac{\mu(B_r(x))}{r^{\theta}} \le \mathcal{H}_{\theta}(E\cap B_r(x)) \le C_2 \frac{\mu(B_r(x))}{r^{\theta}}, \qquad \text{for all } (x, r) \in E\times(0, 1].
    \end{equation}
\end{Def}

\subsection{Besov spaces and hyperbolic fillings}
Throughout this section, we fix a metric measure space $X=(X, d, \mu)$. For brevity, we suppress $\mu$ from the notation for function spaces
on \(X\). Thus, $L_p(X):=L_p(X, \mu)$, $L_p^{\operatorname{loc}}(X):=L_p^{\operatorname{loc}}(X, \mu)$.
\par
For each \(f\in L_1^{\operatorname{loc}}(X)\) and each \(t>0\), we set
\begin{equation}
    \overline{\Delta}_t f(x)
    :=
    \fint\limits_{B_t(x)}
    |f(y)-f_{B_t(x)}|\,d\mu(y).
\end{equation}

\begin{Def}
    Let \(s\in(0,1)\) and \(p,q\in[1,\infty]\). The Besov space
    \(B^s_{p,q}(X)\) is the space of all \(f\in L_p(X)\) such that
    \begin{equation}
        \|f\|_{b^s_{p,q}(X)}
        :=
        \left(
        \int\limits_0^1
        \left(t^{-s}\|\overline{\Delta}_t f\|_{L_p(X)}\right)^q
        \frac{dt}{t}
        \right)^{1/q}
        <\infty,
    \end{equation}
    with the usual modification when \(q=\infty\). We equip this space with
    the norm
    \begin{equation}
        \|f\|_{B^s_{p,q}(X)}
        :=
        \|f\|_{L_p(X)}
        +
        \|f\|_{b^s_{p,q}(X)}, \qquad f\in B^s_{p, q}(X).
    \end{equation}
\end{Def}
We use two simple estimates for the Besov norm (see \cite[Remark~2.14 and Remark~2.16]{chik}). First, if $\mu$ is uniformly locally doubling, then for each $C > 0$, we have an equivalence of norms
\begin{equation}
    \label{eq.bes_norm_eq}
    \|f\|_{L_p(X)} + \left(\sum_{k=0}^{\infty} 2^{ksq}\|\overline{\Delta}_{C2^{-k}}f\|_{L_p(X)}^q\right)^{1/q} \approx \|f\|_{B^s_{p, q}(X)}.
\end{equation}
Second, if $X\in \mathfrak{U}_p$, then for each $f\in \operatorname{LIP}(X)\cap L_p(X)$ and any $\delta \in (0, 1)$
\begin{equation}
\label{eq.lip_bes_norm_est}
    \|f\|_{B^s_{p, q}(X)} \lesssim \delta^{1-s}\|\operatorname{lip}f\|_{L_p(X)}+\delta^{-s}\|f\|_{L_p(X)}.
\end{equation}
\par
We include this standard oscillation definition because it is the one used in
the trace theorem quoted below. However, for the nonlinearity argument it will be more convenient to use a \emph{discrete model} of Besov spaces via \emph{hyperbolic fillings}. This beautiful concept has appeared in the study of various function spaces (see, for example, \cite{bonk_sobolev, soto_liz_tr, soto_bes}). The
setting of these works is slightly different from ours. In particular,
\cite{soto_bes} deals with homogeneous Besov spaces on globally doubling metric
measure spaces. For our purposes, we use a localized inhomogeneous version of Soto's construction. Although our construction is inspired by
\cite{soto_bes}, we provide direct proofs of the estimates used in what
follows. In the present locally doubling, inhomogeneous setting, these
arguments rely only on finite overlap and local scale estimates. We begin with the following packing estimate.
\begin{Lm}[\cite{tyul}, Proposition 2.12]
\label{Lm.loc_d_m}
Let $(X,d,\mu)$ be a metric measure space with uniformly locally doubling
measure. Then, for every $R>0$ and $c\ge1$, every ball $B_{cR}(x)$ contains at
most $N_\mu(R,c)
    =
    C_\mu((c+1)R)^{\log_2(2c)+1}+1$
pairwise disjoint balls of radius $R$.
\end{Lm}
\par
We now construct the \emph{truncated hyperbolic filling} of $X$. For each $n\ge0$, let
$\{\xi_v\}_{v\in V_n}$ be a maximal $2^{-n-1}$-separated subset of $X$, where
$V_n$ is a finite or countable index set. For each $v\in V_n$, we write $|v|=n$
and call $|v|$ the \emph{level} of $v$. Given $n\ge0$ and $v\in V_n$, we set $B(v):=B_{2^{-n}}(\xi_v)$.
By Lemma~\ref{Lm.loc_d_m}, for each $c\ge1$ the balls
$\{cB(v)\}_{v\in V_n}$, $n\ge0$, have uniformly bounded overlap in $X$. Moreover, the balls $\{B_{2^{-n-1}}(\xi_v)\}_{v\in V_n}$ cover $X$, while the balls $\{B_{2^{-n-2}}(\xi_v)\}_{v\in V_n}$ are pairwise disjoint.
\par
Write $V:=\bigsqcup_{n\ge0}V_n$.
We define a \emph{graph} \((V,\mathcal E)\) as follows. Two distinct vertices
\(v,v'\in V\) are connected by an edge if and only if
\begin{equation}
    B(v)\cap B(v')\neq\emptyset
    \qquad\text{and}\qquad
    \big||v|-|v'|\big|\le1.
\end{equation}
We write \(v\sim v'\) if \(v\) and \(v'\) are neighbors, that is, if they are
connected by an edge. For each \(v\in V\), let \(\operatorname{deg}(v)\) denote
the number of its neighbors. By Lemma~\ref{Lm.loc_d_m}, there exists a constant
\(D\in\mathbb N\), depending only on the local doubling constant, such that
\begin{equation}
\label{eq.deg_bound}
    \operatorname{deg}(v)\le D,
    \qquad v\in V.
\end{equation}
\par
Next, we define a discrete sequence space on \(V\). Given \(u:V\to\mathbb R\),
we define its \emph{discrete derivative} \(du: V\to\mathbb{R}^D\) by
\begin{equation}
    du(v)
    :=
    \bigl(
    u(w_1)-u(v),
    u(w_2)-u(v),
    \ldots,
    u(w_{\operatorname{deg}(v)})-u(v),
    0,\ldots, 0
    \bigr),
\end{equation}
where \(w_1,\ldots,w_{\operatorname{deg}(v)}\) are the neighbors of \(v\). We use
the Euclidean length
\begin{equation}
    |du(v)|
    :=
    \left(
    \sum_{v'\sim v}|u(v')-u(v)|^2
    \right)^{1/2}.
\end{equation}
By \eqref{eq.deg_bound}, replacing this Euclidean length by any \(\ell_r\)-length,
\(r\in[1,\infty]\), gives an equivalent quantity.
\begin{Def}
    Let \(s\in(0,1)\) and \(p,q\in[1,\infty]\). We define \(db^s_{p,q}(V)\)
    as the space of all sequences \(u:V\to\mathbb R\) such that
    \begin{equation}
    \begin{split}
        \|u\|_{db^s_{p,q}(V)}
        &:=
        \left(
        \sum_{v\in V_0}\mu(B(v))|u(v)|^p
        \right)^{1/p}
        \\
        &\quad+
        \left(
        \sum_{n=0}^{\infty}
        2^{nsq}
        \left(
        \sum_{v\in V_n}\mu(B(v))|du(v)|^p
        \right)^{q/p}
        \right)^{1/q}
        <\infty
    \end{split}
    \end{equation}
    with the standard modification when $p=\infty$ or $q=\infty$.
\end{Def}
\begin{Remark}
\label{Rm.seq_sp}
    The space \(db^s_{p,q}(V)\) is a Banach space. Moreover, after fixing an
    enumeration of the neighbors of each vertex and identifying all finite
    dimensional spaces \(\mathbb R^{\operatorname{deg}(v)}\) with subspaces of
    \(\mathbb R^D\), the map
    \begin{equation}
         u \to \left(
        \{\mu(B(v))^{1/p}u(v)\}_{v\in V_0},
        \{2^{ns}\mu(B(v))^{1/p}du(v)\}_{v\in V_n,\ n\ge0}
        \right)
    \end{equation}
    defines a linear isomorphic embedding of \(db^s_{p,q}(V)\) into
    \(\ell_q(\ell_p)\). Its range is closed. We denote this embedding by $\mathcal J:db^s_{p,q}(V)\to\ell_q(\ell_p)$.
    In particular, in the case \(q=1\), \(db^s_{p,1}(V)\)
    is linearly isomorphic to a closed subspace of \(\ell_1(\ell_p)\).
\end{Remark}
\par
To relate this definition to the oscillation definition above, for each
\(f\in L_1^{\operatorname{loc}}(X)\) we define its \emph{Poisson extension} by
\begin{equation}
    Pf(v)
    :=
    \fint\limits_{B(v)} f(y)\,d\mu(y),
    \qquad v\in V.
\end{equation}

\begin{Th}
\label{Th.equiv_def_besov_spaces}
    Assume that \((X,d,\mu)\) is a metric measure space with uniformly locally
    doubling measure. Let \(s\in(0,1)\) and \(p,q\in[1,\infty)\). For
    \(f\in L_1^{\operatorname{loc}}(X)\), the following conditions are equivalent:
    \begin{enumerate}
        \item \(f\in B^s_{p,q}(X)\);
        \item \(Pf\in db^s_{p,q}(V)\).
    \end{enumerate}
    Moreover,
    \begin{equation}
        \|f\|_{B^s_{p,q}(X)}
        \approx
        \|Pf\|_{db^s_{p,q}(V)}.
    \end{equation}
\end{Th}
\begin{proof}
    We adapt the arguments of \cite{soto_liz_tr, soto_bes} to the local
    inhomogeneous setting and provide the details for completeness.

    First, assume that \(f\in B^s_{p,q}(X)\). Let \(v\in V_n\), \(n\ge0\), and
    let \(v'\sim v\). If \(x\in B(v)\), then, by the definition of the graph,
    both \(B(v)\) and \(B(v')\) are contained in \(B_{C2^{-n}}(x)\), where \(C\) depends only on the structural constants. Hence, by the local doubling property,
    \begin{equation}
    \begin{split}
        |Pf(v)-Pf(v')|
        &\le
        \fint\limits_{B(v)}
        |f(y)-f_{B_{C2^{-n}}(x)}|\,d\mu(y)
        \\
        &\quad+
        \fint\limits_{B(v')}
        |f(y)-f_{B_{C2^{-n}}(x)}|\,d\mu(y)
        \lesssim
        \overline{\Delta}_{C2^{-n}}f(x).
    \end{split}
    \end{equation}
    Therefore, using the bounded overlap of the balls \(B(v)\), \(v\in V_n\),
    and the uniform bound on the degrees of the graph, we obtain
    \begin{equation}
        \sum_{v\in V_n}\mu(B(v))|d(Pf)(v)|^p
        \lesssim
        \int\limits_X |\overline{\Delta}_{C2^{-n}}f(x)|^p\,d\mu(x).
    \end{equation}
    By Jensen's inequality,
    \begin{equation}
        |Pf(v)|^p \le \fint\limits_{B(v)}|f(x)|^pd\mu(x),
    \end{equation}
    for all $v\in V_0$. Therefore, using the bounded overlap of the balls $B(v)$, $v\in V_0$, we obtain
    \begin{equation}
        \sum_{v\in V_0}\mu(B(v))|Pf(v)|^p
        \lesssim
        \|f\|_{L_p(X)}^p.
    \end{equation}
     Hence, by \eqref{eq.bes_norm_eq},
    \begin{equation}
        \|Pf\|_{db^s_{p,q}(V)}
        \lesssim
        \|f\|_{B^s_{p,q}(X)}.
    \end{equation}
    \par
    Conversely, assume that \(Pf\in db^s_{p,q}(V)\). Since \(f\in
    L_1^{\operatorname{loc}}(X)\), almost every point of \(X\) is a Lebesgue
    point of \(f\). Fix \(k\ge1\), and let \(x_1,x_2\) be Lebesgue points of
    \(f\) such that \(d(x_1,x_2)<2^{-k}\).
    Choose \(v_{k-1}\in V_{k-1}\) such that
    \(x_1\in \frac12 B(v_{k-1})\). Then \(x_2\in B(v_{k-1})\). Thus we may start
    both chains from the same vertex \(v_{k-1}\). For \(i=1,2\) and
    \(n\ge k-1\), choose \(v_n^i\in V_n\) such that \(x_i\in B(v_n^i)\), with
    \(v_{k-1}^1=v_{k-1}^2=v_{k-1}\). Since \(x_i\in B(v_n^i)\cap B(v_{n+1}^i)\),
    the vertices \(v_n^i\) and \(v_{n+1}^i\) are neighbors.

    Since \(x_i\) is a Lebesgue point of \(f\), the local doubling property
    implies
    \begin{equation}
        Pf(v_n^i)\to f(x_i),
        \qquad n\to\infty.
    \end{equation}
    Therefore
    \begin{equation}
        f(x_i)
        =
        Pf(v_{k-1})
        +
        \sum_{n\ge k-1}
        \left(Pf(v_{n+1}^i)-Pf(v_n^i)\right),
        \qquad i=1,2.
    \end{equation}
    \par
    Define
    \begin{equation}
        g_k(x)
        :=
        \sum_{|v|\ge k}
        |d(Pf)(v)|\chi_{B(v)}(x).
    \end{equation}
    Then
    \begin{equation}
        |f(x_1)-f(x_2)|
        \lesssim
        g_k(x_1)+g_k(x_2).
    \end{equation}
    Since \(d(y, z)<2^{-k}\) for \(y, z\in B_{2^{-k-1}}(x)\), the preceding
    estimate gives
\begin{equation}
    \overline{\Delta}_{2^{-k-1}}f(x)
    \lesssim
    \fint\limits_{B_{2^{-k-1}}(x)} g_{k}(y)\,d\mu(y).
\end{equation}
        By Jensen's inequality, Fubini's theorem, and the local doubling property,
    \begin{equation}
    \label{eq.1}
        \|\overline{\Delta}_{2^{-k-1}}f\|_{L_p(X)}^p
        \lesssim
        \int\limits_X |g_k(y)|^p
        \int\limits_{B_{2^{-k-1}}(y)}
        \frac{d\mu(x)}{\mu(B_{2^{-k-1}}(x))}
        d\mu(y)
        \lesssim
        \|g_k\|_{L_p(X)}^p.
    \end{equation}
    Put
\begin{equation}
    A_n
    :=
    \left(
    \sum_{v\in V_n}\mu(B(v))|d(Pf)(v)|^p
    \right)^{1/p}.
\end{equation}
Then, by Minkowski's inequality and the bounded overlap of the balls
    \(B(v)\), \(v\in V_n\), we have
    \begin{equation}
        \|g_k\|_{L_p(X)}
    \lesssim
    \sum_{n=k}^{\infty}A_n.
    \end{equation}
Hence, by the discrete Hardy inequality,
\begin{equation}
\begin{split}
    \left(
    \sum_{k=1}^{\infty}
    2^{ksq}\|g_k\|_{L_p(X)}^q
    \right)^{1/q}
    \lesssim
    \left(
    \sum_{k=1}^{\infty}
    2^{ksq}
    \left(
    \sum_{n=k}^{\infty}A_n
    \right)^q
    \right)^{1/q}
    \lesssim
    \left(
    \sum_{n=1}^{\infty}
    2^{nsq}A_n^q
    \right)^{1/q}.
\end{split}
\end{equation}
    Consequently, using \eqref{eq.bes_norm_eq} and \eqref{eq.1}, we obtain
    \begin{equation}
    \label{eq.ess_est1}
    \begin{split}
        \|f\|_{b^s_{p,q}(X)}
        &\lesssim
        \left(
        \sum_{k=1}^{\infty}
        2^{ksq}\|g_k\|_{L_p(X)}^q
        \right)^{1/q}
        +
        \|f\|_{L_p(X)}
        \\
        &\lesssim
        \left(
        \sum_{n=0}^{\infty}
        2^{nsq}
        \left(
        \sum_{v\in V_n}\mu(B(v))|d(Pf)(v)|^p
        \right)^{q/p}
        \right)^{1/q}
        +
        \|f\|_{L_p(X)}.
    \end{split}
    \end{equation}
    \par
    It remains to estimate the \(L_p\)-norm of \(f\). Let  \(x\in X\) be a Lebesgue point of $f$. Choose \(v_x\in V_0\) such that
    \(x\in \frac12 B(v_x)\). Arguing along a chain starting at \(v_x\), we obtain
\begin{equation}
    |f(x)|
    \lesssim
    |Pf(v_x)|+g_0(x).
\end{equation}
Consequently,
\begin{equation}
    |f(x)|
    \lesssim
    \sum_{v\in V_0}|Pf(v)|\chi_{B(v)}(x)+g_0(x).
\end{equation}
Using the bounded overlap of the balls \(B(v)\), \(v\in V_0\), we get
    \begin{equation}
    \label{eq.ess_est2}
    \begin{split}
        \|f\|_{L_p(X)}
        \lesssim
        \left(
        \sum_{v\in V_0}\mu(B(v))|Pf(v)|^p
        \right)^{1/p}
        +\\
        \left(
        \sum_{n=0}^{\infty}
        2^{nsq}
        \left(
        \sum_{v\in V_n}\mu(B(v))|d(Pf)(v)|^p
        \right)^{q/p}
        \right)^{1/q}.
    \end{split}
    \end{equation}
    Here we used H\"older's inequality in the summation over \(n\) when \(q>1\). The case \(q=1\) is immediate.
    Combining \eqref{eq.ess_est1} and \eqref{eq.ess_est2}, we get
    \begin{equation}
        \|f\|_{B^s_{p,q}(X)}
        \lesssim
        \|Pf\|_{db^s_{p,q}(V)}.
    \end{equation}
    The proof is complete.
\end{proof}
\begin{Ca}
\label{Ca.isom_bes_sp}
    Under the assumptions of Theorem~\ref{Th.equiv_def_besov_spaces},
    the space \(B^s_{p,q}(X)\) is linearly isomorphic to a closed subspace of
    \(db^s_{p,q}(V)\).
\end{Ca}
\begin{proof}
    By Theorem~\ref{Th.equiv_def_besov_spaces}, the Poisson extension
    \(P:B^s_{p,q}(X)\to db^s_{p,q}(V)\) is linear and satisfies
    \begin{equation}
        \|Pf\|_{db^s_{p,q}(V)}
        \approx
        \|f\|_{B^s_{p,q}(X)}.
    \end{equation}
    Hence \(P\) is an isomorphic embedding. It remains only to show that its
    range is closed.

    Let \(Pf_j\to u\) in \(db^s_{p,q}(V)\). Then, by the lower estimate above,
    \begin{equation}
        \|f_j-f_m\|_{B^s_{p,q}(X)}
        \lesssim
        \|Pf_j-Pf_m\|_{db^s_{p,q}(V)}.
    \end{equation}
    Thus \((f_j)\) is a Cauchy sequence in \(B^s_{p,q}(X)\). Since
    \(B^s_{p,q}(X)\) is complete, there exists \(f\in B^s_{p,q}(X)\) such that
    \(f_j\to f\) in \(B^s_{p,q}(X)\). By the boundedness of \(P\), we have
    \(Pf_j\to Pf\) in \(db^s_{p,q}(V)\). Therefore \(u=Pf\), and so
    \(P(B^s_{p,q}(X))\) is closed.
\end{proof}

\section{The main result: \(p\neq2\)}
Throughout this section, let \(p\in[1,\infty)\), let
\(X\in\mathfrak U_p\), and let \(E\subset X\) be Ahlfors--David
codimension-\(\theta\) regular for some \(\theta\in(0,p)\).
By \cite[Corollary~1.5]{chik}, the trace-space of
\(B^{\theta/p}_{p,1}(X)\) to \(E\) is
\(L_p(E,\mathcal H_\theta\lfloor_E)\), with equivalent norms. Moreover, \cite[Theorem~1.4]{chik} provides a bounded nonlinear
extension operator. The goal of this section is to
show that, when \(p\neq2\), such an extension operator cannot be linear.

The proof is based on a Banach-space obstruction. If a bounded linear extension
operator existed, then \(L_p(E, \mathcal{H}_{\theta}\lfloor_E)\) would be isomorphic to a closed subspace of
\(B^{\theta/p}_{p,1}(X)\). By the discrete model from the previous section, the
latter space embeds into \(\ell_1(\ell_p)\). On the other hand, \(L_p(E,\mathcal{H}_{\theta}\lfloor_E)\)
contains a copy of \(\ell_2\). Thus the key point is that \(\ell_2\) cannot be isomorphically embedded into \(\ell_1(\ell_p)\) for \(p\neq2\). This strategy is inspired by the argument used in \cite{kaz_woj} for the
nonexistence of a linear extension operator in the endpoint Sobolev trace
problem.

\begin{Def}
    Given \(p,q\in[1,\infty]\), we denote by \(\ell_q(\ell_p)\) the space of all
    sequences \(x=(x_1,x_2,\ldots)\), where \(x_i\in\ell_p\), such that
    \begin{equation}
        \|x\|_{\ell_q(\ell_p)}
        :=
        \left(
        \sum_{i=1}^{\infty}\|x_i\|_{\ell_p}^q
        \right)^{1/q}
        <\infty,
    \end{equation}
    with the usual modification when \(q=\infty\).
\end{Def}
The following key property is known (see \cite[Proposition~2.3]{Cembr}). For completeness, we give a self-contained proof.
\begin{Th}
\label{Th.keystone_prop}
    Let \(p\in[1,\infty)\), \(p\neq2\). Then \(\ell_2\) cannot be
    isomorphically embedded into \(\ell_1(\ell_p)\).
\end{Th}

\begin{proof}
    Assume, to the contrary, that there exists a bounded linear operator $T:\ell_2\to \ell_1(\ell_p)$
    which is bounded below. Thus, for some constants \(m,M>0\),
    \begin{equation}
        m\|u\|_{\ell_2}
        \le
        \|Tu\|_{\ell_1(\ell_p)}
        \le
        M\|u\|_{\ell_2},
        \qquad u\in\ell_2.
    \end{equation}
\par
    For each \(N\in\mathbb N\), define
    \begin{equation}
        \Pi_Nx=(x_1,\ldots,x_N,0,0,\ldots),
        \qquad x=(x_i)_{i=1}^{\infty}\in\ell_1(\ell_p).
    \end{equation}
    The range of \(\Pi_N\) is \(\ell_1^N(\ell_p)\), which is linearly isomorphic to
    \(\ell_p\). Since \(p\neq2\), the space \(\ell_2\) is not isomorphic to a
    closed subspace of \(\ell_p\) (see \cite[Proposition 2.a.2]{Lind} and the discussion after the proposition). Hence $\Pi_NT:\ell_2\to\ell_1(\ell_p)$
    is not bounded below on any infinite-dimensional subspace of \(\ell_2\).
    Indeed, otherwise it would give an isomorphic embedding of an
    infinite-dimensional Hilbert space, hence of \(\ell_2\), into \(\ell_p\).
\par
    We now construct an orthonormal sequence \(\{u_k\}_{k=1}^{\infty}\subset\ell_2\) and integers $0=N_0<N_1<N_2<\cdots$
    such that \(Tu_k\) is concentrated, up to a small error, on the rows
    \(N_{k-1}+1,\ldots,N_k\). Put $\varepsilon_k=2^{-k}\frac m4$.
    Choose an arbitrary unit vector \(u_1\in\ell_2\). Since \(\Pi_{N_0}=0\), we have
    \begin{equation}
        \|\Pi_{N_0}Tu_1\|_{\ell_1(\ell_p)}=0<\varepsilon_1.
    \end{equation}
    Choose \(N_1>N_0\) so large that
    \begin{equation}
        \|(\operatorname{Id}-\Pi_{N_1})Tu_1\|_{\ell_1(\ell_p)}
        <\varepsilon_1.
    \end{equation}
    Suppose that \(u_1,\ldots,u_{k-1}\) and \(N_1,\ldots,N_{k-1}\) have been
    constructed. Since \(\Pi_{N_{k-1}}T\) is not bounded below on the
    infinite-dimensional subspace
    \begin{equation}
        \{u_1,\ldots,u_{k-1}\}^{\perp}\subset\ell_2,
    \end{equation}
    we can choose a unit vector $u_k\in \{u_1,\ldots,u_{k-1}\}^{\perp}$
    such that
    \begin{equation}
        \|\Pi_{N_{k-1}}Tu_k\|_{\ell_1(\ell_p)}
        <\varepsilon_k.
    \end{equation}
    Since \(Tu_k\in\ell_1(\ell_p)\), we can choose \(N_k>N_{k-1}\) such that
    \begin{equation}
        \|(\operatorname{Id}-\Pi_{N_k})Tu_k\|_{\ell_1(\ell_p)}
        <\varepsilon_k.
    \end{equation}
    \par
    Define $z_k:=(\Pi_{N_k}-\Pi_{N_{k-1}})Tu_k$. Then
    \begin{equation}
        \|Tu_k-z_k\|_{\ell_1(\ell_p)}
        \le
        \|\Pi_{N_{k-1}}Tu_k\|_{\ell_1(\ell_p)}
        +
        \|(\operatorname{Id}-\Pi_{N_k})Tu_k\|_{\ell_1(\ell_p)}
        <2\varepsilon_k.
    \end{equation}
    Since \(\|u_k\|_{\ell_2}=1\) and \(T\) is bounded below, $\|Tu_k\|_{\ell_1(\ell_p)}\ge m$.
    Hence
    \begin{equation}
        \|z_k\|_{\ell_1(\ell_p)}
        \ge
        m-2\varepsilon_k
        \ge
        \frac m2.
    \end{equation}
    \par
    The vectors \(z_k\) have pairwise disjoint row supports. Therefore, for every
    \(l\in\mathbb N\),
    \begin{equation}
        \left\|\sum_{k=1}^l z_k\right\|_{\ell_1(\ell_p)}
        =
        \sum_{k=1}^l\|z_k\|_{\ell_1(\ell_p)}
        \ge
        \frac{lm}{2}.
    \end{equation}
    On the other hand, \(u_1,\ldots,u_l\) are orthonormal in \(\ell_2\), and thus
    \begin{equation}
        \left\|\sum_{k=1}^l u_k\right\|_{\ell_2}
        =
        \sqrt l.
    \end{equation}
    By the boundedness of \(T\),
    \begin{equation}
        \left\|\sum_{k=1}^l Tu_k\right\|_{\ell_1(\ell_p)}
        \le
        M\sqrt l.
    \end{equation}
    However,
    \begin{equation}
    \begin{split}
        \left\|\sum_{k=1}^l Tu_k\right\|_{\ell_1(\ell_p)}
        &\ge
        \left\|\sum_{k=1}^l z_k\right\|_{\ell_1(\ell_p)}
        -
        \sum_{k=1}^l\|Tu_k-z_k\|_{\ell_1(\ell_p)}
        \\
        &\ge
        \frac{lm}{2}
        -
        2\sum_{k=1}^l\varepsilon_k
        \ge
        \frac{lm}{2}-\frac m2.
    \end{split}
    \end{equation}
    Consequently,
    \begin{equation}
        M\sqrt l
        \ge
        \frac{lm}{2}-\frac m2,
        \qquad \text{for all } l\in\mathbb N,
    \end{equation}
    which is impossible as \(l\to\infty\). The proof is complete.
\end{proof}
\par
To prove that the extension operator has to be nonlinear, we must exclude
purely atomic trace-spaces. More precisely, we need to ensure that the trace-space \(L_p(E,\mathcal H_\theta\lfloor_E)\) contains a closed copy of
\(\ell_2\). The standard way to obtain such a copy is to use Rademacher
functions. This requires a nonatomic part of positive measure.
\begin{Lm}
\label{Lm.l2_in_L_p}
     Assume that there exists a measurable set
    \(A\subset E\) such that $0<\mathcal{H}_{\theta}(A)<\infty$
    and \(\mathcal{H}_{\theta}\lfloor_A\) is nonatomic. Then \(L_p(E,\mathcal{H}_{\theta}\lfloor_E)\) contains a closed
    subspace isomorphic to \(\ell_2\), for every \(p\in[1,\infty)\).
\end{Lm}

\begin{proof}
    The subspace
    \begin{equation}
        L_p(A,\mathcal{H}_{\theta}\lfloor_A)
        \simeq
        \{f\in L_p(E,\mathcal{H}_{\theta}\lfloor_E): f=0 \text{ on } E\setminus A\}
    \end{equation}
    is closed in \(L_p(E,\mathcal{H}_{\theta}\lfloor_E)\). Since \(\mathcal{H}_{\theta}\lfloor_A\) is finite and nonatomic,
    one can construct Rademacher functions on \(A\). By Khintchine's inequality,
    their closed linear span is isomorphic to \(\ell_2\). Hence \(L_p(E,\mathcal{H}_{\theta}\lfloor_E)\)
    contains a closed copy of \(\ell_2\).
\end{proof}
\begin{Remark}
    In particular, the assumption of the lemma is satisfied if there exist
    \(x_0\in E\) and \(r_0>0\) such that
    \begin{equation}
        0<\mathcal{H}_\theta(B_{r_0}(x_0)\cap E)<\infty
    \end{equation}
    and \(\mathcal{H}_\theta\lfloor_{B_{r_0}(x_0)\cap E}\) is nonatomic.
\end{Remark}
\begin{Remark}
    A sufficient condition for the measure
    \(\mathcal{H}_\theta\lfloor_{B_{r_0}(x_0)\cap E}\) to be nonatomic is
    \begin{equation}
        \lim_{r\to0}\frac{\mu(B_r(x))}{r^\theta}=0
        \qquad
        \text{for every }x\in B_{r_0}(x_0)\cap E.
    \end{equation}
    Indeed, by the upper Ahlfors--David codimension-\(\theta\) estimate,
    \begin{equation}
        \mathcal{H}_\theta(\{x\})
        \le
        \mathcal{H}_\theta(B_r(x)\cap E)
        \lesssim
        \frac{\mu(B_r(x))}{r^\theta}
        \to0.
    \end{equation}
\end{Remark}
In practice, the nonatomicity assumption can be verified using a reverse
volume decay estimate. The results
\cite[Lemma~8.1.13, Proposition~8.1.6, and
Remark~8.1.15]{sob_mms}, with the corresponding local modifications,
imply that for every \(R>0\) there exist
\(q=q(R)>0\) and \(C=C(R,q)>0\) such that
\begin{equation}
    \frac{\mu(B_{r'}(x'))}{\mu(B_r(x))}
    \le
    C\left(\frac{r'}r\right)^q
\end{equation}
whenever \(B_{r'}(x')\subset B_r(x)\) and \(0<r'\le r\le R\). Let \(q_\mu(R)\) be the supremum of all exponents \(q\) for which this
estimate holds. If \(\theta<q_\mu(R)\), choose
\(q\in(\theta,q_\mu(R))\). Applying the estimate to
\(B_r(x)\subset B_R(x)\), we obtain
\begin{equation}
    \frac{\mu(B_r(x))}{r^\theta}
    \lesssim
    \frac{\mu(B_R(x))}{R^q}\,r^{q-\theta}
    \longrightarrow0
    \qquad\text{as }r\to0.
\end{equation}
In particular, if $(X, d, \mu)$ is $Q$-regular, i.e., $\mu(B_r(x))\approx r^Q$, for all $x\in X$ and $r\in (0, 1]$, and $E$ is Ahlfors--David codimension-$\theta$ regular with $\theta\in (0, Q)$, then $\mathcal{H}_{\theta}\lfloor_E$ is nonatomic.
\begin{Th}
\label{Th.no_linear_ext_metric_pneq2}
    Let \(p\in[1,\infty)\), \(p\neq2\), and let \(\theta\in(0,p)\).
    Let \(X=(X,d,\mu)\) be a \(p\)-admissible metric measure space, and let
    \(E\subset X\) be an Ahlfors--David codimension-\(\theta\) regular set.
   Assume, in addition, that there exists a measurable set \(A \subset E\) such that $0<\mathcal{H}_{\theta}(A)<\infty$ and \(\mathcal{H}_{\theta}\lfloor_A\) is nonatomic. Then there is no bounded linear
    extension operator
    \begin{equation}
        \operatorname{Ext}:
        L_p(E,\mathcal H_\theta\lfloor_E)
        \longrightarrow
        B^{\theta/p}_{p,1}(X).
    \end{equation}
\end{Th}
\begin{proof}
    Suppose, to the contrary, that such a bounded linear extension operator
    exists. Since
    \begin{equation}
        \operatorname{Tr}\circ\operatorname{Ext}
        =
        \operatorname{Id}_{L_p(E,\mathcal H_\theta\lfloor_E)},
    \end{equation}
    the operator \(\operatorname{Ext}\) is bounded below. Indeed, by the boundedness
    of the trace operator,
    \begin{equation}
        \|\phi\|_{L_p(E,\mathcal H_\theta\lfloor_E)}
        =
        \|\operatorname{Tr}\operatorname{Ext}\phi\|_{L_p(E,\mathcal H_\theta\lfloor_E)}
        \le
        \|\operatorname{Tr}\|
        \|\operatorname{Ext}\phi\|_{B^{\theta/p}_{p,1}(X)}.
    \end{equation}
    Hence \(L_p(E,\mathcal H_\theta\lfloor_E)\) is isomorphic to a closed subspace
    of \(B^{\theta/p}_{p,1}(X)\).

    By Corollary~\ref{Ca.isom_bes_sp}, with \(s=\theta/p\) and \(q=1\),
    \(B^{\theta/p}_{p,1}(X)\) is linearly isomorphic to a closed subspace of
    \(db^{\theta/p}_{p,1}(V)\). Moreover, by Remark~\ref{Rm.seq_sp}, the latter space embeds isomorphically into
    \(\ell_1(\ell_p)\). Therefore \(L_p(E,\mathcal{H}_{\theta}\lfloor_E)\) is isomorphic to a closed
    subspace of \(\ell_1(\ell_p)\).

    By Lemma~\ref{Lm.l2_in_L_p}, \(L_p(E,\mathcal{H}_{\theta}\lfloor_E)\) contains a closed subspace isomorphic to
    \(\ell_2\). Consequently, \(\ell_2\) is isomorphic to a closed subspace of
    \(\ell_1(\ell_p)\), contradicting Theorem~\ref{Th.keystone_prop}.
\end{proof}
\begin{Remark}
    The additional assumption that \(L_p(E,\mathcal H_\theta\lfloor_E)\) contains a copy of \(\ell_2\)
    cannot simply be omitted. If the trace measure is supported on finitely many
    points, then the trace-space is finite-dimensional and a bounded linear
    extension operator may exist.
    For example, let \(E=\{x_1,\ldots,x_N\}\) and assume that
    \(0<\mathcal H_\theta(\{x_j\})<\infty\) for all \(j\). Then
    \(L_p(E,\mathcal H_\theta\lfloor_E)\) is finite-dimensional. Choose Lipschitz functions
    \(\psi_1,\ldots,\psi_N\in B^{\theta/p}_{p,1}(X)\) such that $\psi_j(x_i)=\delta_{ij}$.
    Then
    \begin{equation}
        \operatorname{Ext}\phi
        :=
        \sum_{j=1}^N \phi(x_j)\psi_j.
    \end{equation}
    defines a bounded linear right inverse to the trace operator.
\end{Remark}
\section{The main result: \(p=2\)}

When \(p=2\), the argument from the previous section no longer works. Indeed,
\(\ell_2\) embeds into \(\ell_1(\ell_2)\), for instance as the first row. However,
every isomorphic embedding of \(\ell_2\) into \(\ell_1(\ell_2)\) has a finite-row
approximation property.

\begin{Lm}
\label{Lm.keyst_prop_p2}
    Let \(H\) be a reflexive Banach space and let \(T:H\to \ell_1(\ell_2)\) be bounded and linear. Then
    \begin{equation}
        \|(\operatorname{Id}-\Pi_N)T\|_{H\to\ell_1(\ell_2)}
        \to0
        \qquad\text{as }N\to\infty.
    \end{equation}
\end{Lm}

\begin{proof}
    Since \(H\) is reflexive, its closed unit ball is weakly compact. Since
    \(T\) is bounded and linear, it is weakly continuous. Hence $K:=T(B_{H})$
    is a weakly compact subset of \(\ell_1(\ell_2)\).
    \par
    We identify \(\ell_1(\ell_2)\) isometrically with \(L^1(\mathbb N,\nu;\ell_2)\), where \(\mathbb N=\{1,2,\ldots\}\) and \(\nu(\{j\})=2^{-j}\). Indeed, the map 
    \begin{equation}
        U:\ell_1(\ell_2) \to L^1(\mathbb N,\nu;\ell_2)
    \end{equation}
    defined by $Ux(j)=2^j x_j$ is an isometric isomorphism. Since weak compactness is preserved under Banach-space isomorphisms, \(U(K)\) is weakly compact in \(L^1(\mathbb N,\nu;\ell_2)\). By the weak compactness criterion in Bochner \(L^1\)-spaces over probability spaces (see \cite[Chapter~IV, Section~2, Theorem~1]{diestel_uhl}), \(U(K)\) is uniformly integrable. Hence, for every \(\varepsilon>0\), there exists \(\delta>0\) such that
\begin{equation}
\nu(A)<\delta
\quad\Longrightarrow\quad
\sup_{x\in K}\sum_{j\in A}\|x_j\|_{\ell_2}<\varepsilon.
\end{equation}
Taking \(A=\{j>N\}\), we obtain
\begin{equation}
\sup_{x\in K}\sum_{j>N}\|x_j\|_{\ell_2}\to0.
\end{equation}
    Equivalently,
    \begin{equation}
        \|(\operatorname{Id}-\Pi_N)T\|_{H\to\ell_1(\ell_2)}
        =
        \sup_{\|u\|_{H}\le1}
        \|(\operatorname{Id}-\Pi_N)Tu\|_{\ell_1(\ell_2)}
        \to0.
    \end{equation}
\end{proof} 
To complete the argument, one needs a finite-scale obstruction. Roughly speaking, if a function is reconstructed only from level-\(N\) hyperbolic filling coefficients, then its restriction to a bounded part of \(E\) belongs to a finite-dimensional space generated by the level-\(N\) partition of unity. Such functions cannot approximate the identity on an infinite-dimensional \(L_2\)-space in operator norm.
\par
However, we cannot use the projection $\Pi_N$ directly. Indeed, the range of \(\Pi_N\) is not contained in the image of \(db^s_{2,1}(V)\) under the embedding into \(\ell_1(\ell_2)\). To proceed in this direction, we fix a collection of nonnegative Lipschitz functions $\{\psi_v\}_{v\in V}$ such that
\begin{enumerate}
\item 
\begin{equation}
\operatorname{supp} \psi_v \subset B(v), \qquad \operatorname{Lip}(\psi_v)\lesssim2^{|v|}, \qquad \text{for all } v\in V;
\end{equation}
\item  for each $n\ge 0$,
\begin{equation}\sum_{v\in V_n}\psi_v(x) = 1, \qquad \text{for all } x\in X.
\end{equation}
\end{enumerate}
Given $n\ge0$, define $S_n:db^s_{p, 1}(V)\to B^s_{p, 1}(X)$ by setting
\begin{equation}
S_nu := \sum_{v\in V_n}u(v)\psi_v, \qquad u\in db^s_{p, 1}(V).
\end{equation}
For a more detailed discussion of these operators, see \cite{soto_liz_tr, soto_bes}. We prove the key facts only for the case $q=1$, since it is the only one used below.
\par
The following tail estimate is the only reconstruction estimate needed in the proof of Theorem~\ref{Th.main_stat}. We give a direct proof, independent of the global reconstruction estimates from \cite{soto_bes}. In the present admissible setting the Poincar\'e inequality allows us to estimate the Besov norm of each level difference by elementary Lipschitz bounds.
\begin{Lm}
\label{Lm.level_reconstruction_tail}
    Let \(s\in(0,1)\), \(p\in[1,\infty)\), and let \(f\in B^s_{p,1}(X)\).
    Assume that $X \in \mathfrak{U}_p$. Then, for every \(N\in\mathbb N\),
    \begin{equation}
        \|f-S_N(Pf)\|_{B^s_{p,1}(X)}
        \le C
        \sum_{n=N}^{\infty}
        2^{ns}
        \left(
        \sum_{v\in V_n}\mu(B(v))|d(Pf)(v)|^p
        \right)^{1/p},
    \end{equation}
    where the constant $C$ is independent of $f$.
\end{Lm}
\begin{proof}
\emph{Step 1}. Fix $u \in db^s_{p, 1}(V)$, \(n\in \mathbb{N}_0\), \(v\in V_n\), and \(x\in B(v)\). Since the family $\{\psi_{v'}\}_{v'\in V_n}$ forms a partition of unity, we have
\begin{equation}
\label{eq.2}
S_{n+1}u(x)-S_nu(x) = \sum_{v'\in V_{n+1}}(u(v')-u(v))\psi_{v'}(x)-\sum_{v'\in V_{n}}(u(v')-u(v))\psi_{v'}(x).
\end{equation}
Whenever \(\psi_{v'}(x)\neq0\), we have \(x\in B(v')\), and hence
\(v'\sim v\) unless \(v'=v\). Therefore, by the uniform bound on the
degrees of the graph,
\begin{equation}
|S_{n+1}u(x)-S_nu(x)| \lesssim|du(v)|.
\end{equation}
Consequently, by the bounded overlap of the balls $\{B(v)\}_{v\in V_n}$,
\begin{equation}
\|S_{n+1}u-S_nu\|_{L_p(X)} \lesssim \left(\sum_{v\in V_n}\mu(B(v))|du(v)|^p\right)^{1/p}.
\end{equation}
Using the same representation \eqref{eq.2} and the bounds
\(\operatorname{Lip}\psi_{v'}\lesssim2^n\), we obtain
\begin{equation}
\|\operatorname{lip}(S_{n+1}u-S_nu)\|_{L_p(X)} \lesssim 2^n\left(\sum_{v\in V_n}\mu(B(v))|du(v)|^p \right)^{1/p}.
\end{equation}
Applying \eqref{eq.lip_bes_norm_est} with \(\delta=2^{-n}\), we get
\begin{equation}
\label{eq.3}
\|S_{n+1}u-S_nu\|_{B^s_{p, 1}(X)} \lesssim 2^{ns}\left(\sum_{v\in V_n}\mu(B(v))|du(v)|^p\right)^{1/p}.
\end{equation}
Since $u\in db_{p, 1}^s(V)$, summing the right-hand side of
\eqref{eq.3} over \(n\) gives
\begin{equation}
    \sum_{n=0}^{\infty}\|S_{n+1}u-S_nu\|_{B^s_{p,1}(X)}
<\infty
\end{equation}
and therefore the telescoping series converges in \(B^s_{p,1}(X)\).
Moreover,
\begin{equation}
\sum_{n=0}^{\infty}
\|S_{n+1}u-S_nu\|_{L_p(X)}
<\infty,
\end{equation}
so the series converges absolutely for \(\mu\)-almost every \(x\in X\).
By the same argument, $S_0u \in B^s_{p, 1}(X)$. Therefore, the limit
\begin{equation}
S u :=\lim_{n \to \infty}S_nu = S_0u+\sum_{n=0}^{\infty}\left(S_{n+1}u-S_nu\right)
\end{equation}
exists in $B^s_{p,1}(X)$ and pointwise $\mu$-almost everywhere (see also \cite[Lemma 2.3]{soto_bes}). 
\par
\emph{Step 2}. Let \(f\in L_1^{\operatorname{loc}}(X)\), and let \(x\) be a Lebesgue point of \(f\). If \(\psi_v(x)\neq0\) and
\(v\in V_n\), then \(x\in B(v)\). Hence, by the local doubling property of $\mu$, 
\begin{equation}
    |Pf(v)-f(x)|\lesssim\fint\limits_{B_{2^{-n+1}}(x)}|f(y)-f(x)|d\mu(y)\to0
\end{equation}
uniformly over the finitely many vertices \(v\in V_n\) for which
\(\psi_v(x)\neq0\). Since \(\sum_{v\in V_n}\psi_v(x)=1\), it follows that
\begin{equation}
    S_n(Pf)(x)\to f(x), \qquad \text{as } n\to \infty
\end{equation}
for \(\mu\)-almost every \(x\in X\). Thus $S(Pf) = f$ pointwise $\mu$-almost everywhere (see also \cite[Theorem 2.5]{soto_bes}).
\par
\emph{Step 3}. Now assume that $u = Pf$, where $f\in B^s_{p, 1}(X)$. Since $S(Pf) = f$,  for each $N\in \mathbb{N}$ and $\mu$-a.e. $x\in X$, we have
\begin{equation}
    f(x) = SPf(x) =  S_NPf(x)+\sum_{n=N}^{\infty}(S_{n+1}Pf(x)-S_nPf(x)).
\end{equation}
Therefore, by \eqref{eq.3}
\begin{equation}
    \|f-S_NPf\|_{B^s_{p, 1}(X)} \lesssim \sum_{n=N}^{\infty}2^{ns}\left(\sum_{v\in V_n}\mu(B(v))|d(Pf)(v)|^p\right)^{1/p}.
\end{equation}
The proof is complete.
\end{proof}
\begin{Th}
    \label{Th.main_p_2}
    Let $X=(X, d, \mu)\in \mathfrak{U}_2$, and let $E\subset X$ be Ahlfors--David codimension-$\theta$ regular, $\theta \in (0, 2)$. Assume that there is $A\subset E$ with $0<\mathcal{H}_{\theta}(A)<\infty$ such that $\mathcal{H}_{\theta}\lfloor_A$ is nonatomic. Then there is no bounded linear extension operator
    \begin{equation}
        \operatorname{Ext}:L_2(E, \mathcal{H}_{\theta}\lfloor_E)\to B^{\theta/2}_{2, 1}(X).
    \end{equation}
\end{Th}
\begin{proof}
Throughout the proof, we write $L_2(E) := L_2(E, \mathcal{H}_{\theta}\lfloor_E)$ and $L_2(A) := L_2(A, \mathcal{H}_{\theta}\lfloor_A)$.
\par
    Assume, to the contrary, that there exists a bounded linear extension operator
    \begin{equation}
        \operatorname{Ext}:L_2(E)\to B^{\theta/2}_{2, 1}(X).
    \end{equation}
    Fix $A\subset E$ with $0<\mathcal{H}_{\theta}(A)<\infty$ such that $\mathcal{H}_{\theta}\lfloor_A$ is nonatomic. Replacing \(A\) by \(A\cap B_R(x_0)\) for a suitable \(R>0\), we may assume that \(A\) is bounded. Indeed, since \(A=\bigcup_R(A\cap B_R(x_0))\) and \(\mathcal H_\theta(A)>0\), one of these intersections has positive measure; the restriction of a nonatomic measure remains nonatomic.
    Then,
    \begin{equation}
         L_2(A) \hookrightarrow L_2(E) \xrightarrow{\operatorname{Ext}} B^{\theta/2}_{2, 1}(X) \xrightarrow{P}db^{\theta/2}_{2, 1}(V) \xrightarrow{\mathcal{J}}\ell_1(\ell_2)
    \end{equation}
    defines an isomorphic embedding of $L_2(A)$ into $\ell_1(\ell_2)$. Define, for each $N\in \mathbb{N}$
    \begin{equation}
        R_N := \operatorname{Tr}\circ S_N \circ P \circ \operatorname{Ext}:L_2(E)\to L_2(E).
    \end{equation}
     Define operators
    \begin{equation}
        I_A: L_2(A)\to L_2(E), \qquad R_A:L_2(E)\to L_2(A),
    \end{equation}
    where where \(I_A\) denotes extension by zero and \(R_A\) denotes restriction to \(A\). By Lemma~\ref{Lm.keyst_prop_p2}, we have
    \begin{equation}
        \|(\operatorname{Id}-\Pi_N)\mathcal{J}P\operatorname{Ext}I_A\|_{L_2(A)\to \ell_1(\ell_2)} \to 0, \qquad\text{as } N\to\infty.
    \end{equation}
    Furthermore, by Lemma~\ref{Lm.level_reconstruction_tail}, we obtain, for each $\phi\in L_2(E)$
    \begin{equation}
        \|\operatorname{Ext}\phi - S_{N}(P\operatorname{Ext}\phi)\|_{B^{\theta/2}_{2, 1}(X)} \lesssim \|(\operatorname{Id}- \Pi_{N+1})\mathcal{J}P\operatorname{Ext}(\phi)\|_{\ell_1(\ell_2)}.
    \end{equation}
    Let 
    \begin{equation}
        F_N:=R_A\circ R_N\circ I_A:L_2(A) \to L_2(A).
    \end{equation}
    Since \(R_A\operatorname{Tr}\operatorname{Ext}I_A=\operatorname{Id}_{L_2(A)}\), we have
\begin{equation}
F_N-\operatorname{Id}_{L_2(A)}
=
R_A\operatorname{Tr}
\bigl(S_NP\operatorname{Ext}I_A-\operatorname{Ext}I_A\bigr).
\end{equation}
Hence
\begin{equation}
\|F_N-\operatorname{Id}_{L_2(A)}\|
\le
\|R_A\|\|\operatorname{Tr}\|\,
\|S_NP\operatorname{Ext}I_A-\operatorname{Ext}I_A\|_{L_2(A)\to B^{\theta/2}_{2,1}(X)}
\to0.
\end{equation}
    \par
    For each $\phi \in L_2(E)$, \(S_N(P\operatorname{Ext}\phi)\) is locally Lipschitz, hence continuous. Therefore, the trace of $S_N(P\operatorname{Ext}\phi)$ coincides with the pointwise restriction of this function to $E$. Moreover,
    \begin{equation}
        S_N(P\operatorname{Ext}\phi)\big|_A \in \operatorname{span}\{\psi_v\big|_{A}\}_{v\in V_N}.
    \end{equation}
    Since \(A\) is bounded, Lemma~\ref{Lm.loc_d_m} implies that only finitely many balls \(B(v)\), \(v\in V_N\), meet \(A\). Thus the space above is finite-dimensional, and \(F_N\) has finite rank. For all sufficiently large \(N\), \(\|F_N-\operatorname{Id}_{L_2(A)}\|<1\). Hence \(F_N\) is invertible by the Neumann series. This is impossible, since \(F_N\) has finite rank while \(L_2(A)\) is infinite-dimensional. This contradiction completes the proof.
\end{proof}

\renewcommand{\bibname}{References}
\bibliographystyle{plain}
\bibliography{refs}

@book{Lind,
  author    = {Lindenstrauss, Joram and Tzafriri, Lior},
  title     = {Classical Banach Spaces I: Sequence Spaces},
  series    = {Ergebnisse der Mathematik und ihrer Grenzgebiete},
  volume    = {92},
  publisher = {Springer-Verlag},
  address   = {Berlin--New York},
  year      = {1977},
  mrnumber  = {0500056}
}

@book{diestel_uhl,
  author    = {Diestel, Joseph and Uhl, Jr., John Jerry},
  title     = {Vector Measures},
  series    = {Mathematical Surveys and Monographs},
  volume    = {15},
  publisher = {American Mathematical Society},
  address   = {Providence, RI},
  year      = {1977},
  isbn      = {978-0-8218-1515-1},
  url       = {https://www.ams.org/books/surv/015/},
  note      = {\href{https://www.ams.org/books/surv/015/}{AMS, Mathematical Surveys and Monographs, Vol.~15}}
}

@misc{soto_bes,
  author        = {Soto, Tom{\'a}s},
  title         = {{Besov} spaces via hyperbolic fillings},
  year          = {2016},
  eprint        = {1606.08082},
  archivePrefix = {arXiv},
  primaryClass  = {math.CA},
  doi           = {10.48550/arXiv.1606.08082},
  url           = {https://arxiv.org/abs/1606.08082},
  note          = {\href{https://arxiv.org/abs/1606.08082}{arXiv:1606.08082}}
}

@article{soto_liz_tr,
  author        = {Bonk, Mario and Saksman, Eero and Soto, Tom{\'a}s},
  title         = {{Triebel--Lizorkin} spaces on metric spaces via hyperbolic fillings},
  journal       = {Indiana Univ. Math. J.},
  volume        = {67},
  number        = {4},
  pages         = {1625--1663},
  year          = {2018},
  doi           = {10.1512/iumj.2018.67.7282},
  eprint        = {1411.5906},
  archivePrefix = {arXiv},
  primaryClass  = {math.CA},
  url           = {https://doi.org/10.1512/iumj.2018.67.7282},
  note          = {\href{https://doi.org/10.1512/iumj.2018.67.7282}{doi: 10.1512/iumj.2018.67.7282}}
}

@article{bonk_sobolev,
  author        = {Bonk, Mario and Saksman, Eero},
  title         = {{Sobolev} spaces and hyperbolic fillings},
  journal       = {Journal f{\"u}r die Reine und Angewandte Mathematik},
  volume        = {737},
  pages         = {161--187},
  year          = {2018},
  doi           = {10.1515/crelle-2015-0036},
  eprint        = {1408.3642},
  archivePrefix = {arXiv},
  primaryClass  = {math.CV},
  url           = {https://doi.org/10.1515/crelle-2015-0036},
  note          = {\href{https://doi.org/10.1515/crelle-2015-0036}{doi: 10.1515/crelle-2015-0036}}
}

@article{kaz_woj,
  author        = {Kazaniecki, Krystian and Wojciechowski, Micha{\l}},
  title         = {Trace operator on von {Koch}'s snowflake},
  journal       = {Potential Anal.},
  volume        = {61},
  pages         = {659--684},
  year          = {2024},
  doi           = {10.1007/s11118-024-10124-w},
  eprint        = {1903.01100},
  archivePrefix = {arXiv},
  primaryClass  = {math.FA},
  url           = {https://doi.org/10.1007/s11118-024-10124-w},
  note          = {\href{https://doi.org/10.1007/s11118-024-10124-w}{doi: 10.1007/s11118-024-10124-w}}
}

@article{tyul,
  author        = {Tyulenev, Alexander I.},
  title         = {Traces of {Sobolev} spaces to irregular subsets of metric measure spaces},
  journal       = {Sb. Math.},
  volume        = {214},
  number        = {9},
  pages         = {1241--1320},
  year          = {2023},
  doi           = {10.4213/sm9893e},
  eprint        = {2212.11271},
  archivePrefix = {arXiv},
  primaryClass  = {math.FA},
  url           = {https://doi.org/10.4213/sm9893e},
  note          = {\href{https://doi.org/10.4213/sm9893e}{doi: 10.4213/sm9893e}}
}

@misc{chik,
  author        = {Chikalov, Aleksei Y.},
  title         = {Traces of {Besov} spaces to regular subsets of metric measure spaces: the limiting case},
  year          = {2026},
  eprint        = {2606.29396},
  archivePrefix = {arXiv},
  primaryClass  = {math.FA},
  url           = {https://arxiv.org/abs/2606.29396},
  note          = {\href{https://arxiv.org/abs/2606.29396}{arXiv:2606.29396}}
}

@article{Han,
  author    = {Han, Yongsheng and M{\"u}ller, Detlef and Yang, Dachun},
  title     = {A theory of {Besov} and {Triebel--Lizorkin} spaces on metric measure spaces modeled on {Carnot--Carath{\'e}odory} spaces},
  journal   = {Abstract and Applied Analysis},
  volume    = {2008},
  pages     = {Article ID 893409, 250 pp.},
  year      = {2008},
  doi       = {10.1155/2008/893409},
  url       = {https://doi.org/10.1155/2008/893409},
  note      = {\href{https://doi.org/10.1155/2008/893409}{doi: 10.1155/2008/893409}}
}

@article{Shanm,
  author  = {Gogatishvili, Amiran and Koskela, Pekka and Shanmugalingam, Nageswari},
  title   = {Interpolation properties of {Besov} spaces defined on metric spaces},
  journal = {Math. Nachr.},
  volume  = {283},
  number  = {2},
  pages   = {215--231},
  year    = {2010},
  doi     = {10.1002/mana.200810242},
  url     = {https://doi.org/10.1002/mana.200810242},
  note    = {\href{https://doi.org/10.1002/mana.200810242}{doi: 10.1002/mana.200810242}}
}

@article{AKZ,
  author        = {Gogatishvili, Amiran and Koskela, Pekka and Zhou, Yuan},
  title         = {Characterizations of {Besov} and {Triebel--Lizorkin} spaces on metric measure spaces},
  journal       = {Forum Math.},
  volume        = {25},
  number        = {4},
  pages         = {787--819},
  year          = {2013},
  doi           = {10.1515/form.2011.135},
  eprint        = {1106.2561},
  archivePrefix = {arXiv},
  primaryClass  = {math.CA},
  url           = {https://doi.org/10.1515/form.2011.135},
  note          = {\href{https://doi.org/10.1515/form.2011.135}{doi: 10.1515/form.2011.135}}
}

@article{Koskela,
  author        = {Koskela, Pekka and Yang, Dachun and Zhou, Yuan},
  title         = {Pointwise characterizations of {Besov} and {Triebel--Lizorkin} spaces and quasiconformal mappings},
  journal       = {Adv. Math.},
  volume        = {226},
  number        = {4},
  pages         = {3579--3621},
  year          = {2011},
  doi           = {10.1016/j.aim.2010.10.020},
  eprint        = {1004.5507},
  archivePrefix = {arXiv},
  primaryClass  = {math.CA},
  url           = {https://doi.org/10.1016/j.aim.2010.10.020},
  note          = {\href{https://doi.org/10.1016/j.aim.2010.10.020}{doi: 10.1016/j.aim.2010.10.020}}
}

@article{gagl,
  author   = {Gagliardo, Emilio},
  title    = {Caratterizzazioni delle tracce sulla frontiera relative ad alcune classi di funzioni in {$n$} variabili},
  journal  = {Rendiconti del Seminario Matematico della Universit{\`a} di Padova},
  volume   = {27},
  pages    = {284--305},
  year     = {1957},
  url      = {https://eudml.org/doc/106977},
  language = {Italian}
}

@article{Peetre,
  author   = {Peetre, Jaak},
  title    = {A counter-example connected with {Gagliardo}'s trace theorem},
  journal  = {Commentationes Mathematicae. Special Issue},
  pages    = {277--282},
  year     = {1979},
  issn     = {0373-8299},
  language = {English},
  zbl      = {0442.46026}
}

@article{Bes_or,
  author   = {Besov, O. V.},
  title    = {Investigation of a family of function spaces in connection with theorems of imbedding and extension},
  journal  = {Translations of the American Mathematical Society. Series 2},
  volume   = {40},
  pages    = {85--126},
  year     = {1964},
  doi      = {10.1090/trans2/040/03},
  url      = {https://doi.org/10.1090/trans2/040/03},
  language = {English},
  note     = {\href{https://doi.org/10.1090/trans2/040/03}{doi: 10.1090/trans2/040/03}}
}

@article{bur_gold,
  author   = {Burenkov, V. I. and Gol'dman, M. L.},
  title    = {On the extension of functions from {$L_p$}},
  journal  = {Trudy Matematicheskogo Instituta Imeni V. A. Steklova},
  volume   = {150},
  pages    = {31--51},
  year     = {1979},
  language = {Russian}
}

@article{Gol,
  author  = {Gol'dman, M. L.},
  title   = {On the extension of functions from {$L_p(\mathbb R^m)$} to a space of higher dimension},
  journal = {Math. Notes},
  volume  = {25},
  number  = {4},
  pages   = {266--270},
  year    = {1979},
  doi     = {10.1007/BF01688477},
  url     = {https://doi.org/10.1007/BF01688477},
  note    = {Russian original: Mat. Zametki 25 (1979), no. 4, 513--520; \href{https://doi.org/10.1007/BF01688477}{doi: 10.1007/BF01688477}}
}

@book{sob_mms,
  author    = {Heinonen, Juha and Koskela, Pekka and Shanmugalingam, Nageswari and Tyson, Jeremy T.},
  title     = {{Sobolev} Spaces on Metric Measure Spaces: An Approach Based on Upper Gradients},
  series    = {New Mathematical Monographs},
  publisher = {Cambridge University Press},
  address   = {Cambridge},
  year      = {2015},
  doi       = {10.1017/CBO9781316135914},
  url       = {https://doi.org/10.1017/CBO9781316135914},
  note      = {\href{https://doi.org/10.1017/CBO9781316135914}{doi: 10.1017/CBO9781316135914}}
}

@article{Cembr,
  author  = {Cembranos, Pilar and Mendoza, Jos{\'e}},
  title   = {On the mutually non isomorphic {$\ell_p(\ell_q)$} spaces},
  journal = {Math. Nachr.},
  volume  = {284},
  number  = {16},
  pages   = {2013--2023},
  year    = {2011},
  doi     = {10.1002/mana.201010056},
  url     = {https://doi.org/10.1002/mana.201010056},
  note    = {\href{https://doi.org/10.1002/mana.201010056}{doi: 10.1002/mana.201010056}}
}
\end{document}